\documentclass[12pt]{amsart}
\usepackage{graphicx}
\usepackage{amssymb}
\usepackage{amsmath}
\usepackage{amsthm}
\usepackage{color}
\newenvironment{purple}{\color{magenta}}{}

\definecolor{darkgreen}{rgb}{0,0.5,0}

\def\boldit#1{\textit{\textbf{#1}}}

\def\qed{\hfill$\square$}
\numberwithin{equation}{section}

\newtheorem{lem}[equation]{\sc Lemma}
\newtheorem{cor}[equation]{\sc Corollary}
\newtheorem{prop}[equation]{\sc Proposition}

\newtheoremstyle{notation}{3pt}{3pt}{}{}{\itshape}{:}{.5em}{\thmname{#1}}
\theoremstyle{notation}

\newtheorem{defin}{\it Definition}

\input prepictex   \input pictex    
\input postpictex

\newcounter{boxsize}
\newcounter{tempcounter}
\def\arr#1#2{\arrow <2mm> [0.25,0.75] from #1 to #2}

\def\ssize{\scriptscriptstyle}

\newcommand\boxes[2]{\ifthenelse{#2=3}{$\scriptstyle P_2^{#1}$}{%
                                       $\scriptstyle P_{#2}^{#1}$}}

  \newcommand\F{{\sf F}}
  \newcommand\A{{\sf A}}
  \newcommand\C{{\sf C}}
  \newcommand\E{{\sf E}}
  \newcommand\G{{\sf G}}
  \newcommand\B{{\sf B}}
  \newcommand\D{{\sf D}}
  \newcommand\p{$'$}
  \renewcommand\c{$_\prime$}
  \newcommand\s{$\sharp$}
  \newcommand\f{$\flat$}

  \newcommand\LL{{\sf L}}
  \newcommand\PP{{\sf P}}
  \newcommand\RR{{\sf R}}

\begin{document}
  
  \begin{center}
    {\large\bf From Schritte and Wechsel to Coxeter Groups}

    \bigskip
    {\normalsize Markus Schmidmeier\footnote{Florida Atlantic University, Boca Raton, FL 33431-0991, {\sf markus@math.fau.edu}}}
  
 \bigskip

  \begin{minipage}{12cm}
    \footnotesize
  {\bf Abstract:} The  PLR-moves of neo-Riemannian theory, when considered 
  as reflections on the edges of an equilateral triangle, define 
  the Coxeter group~$\widetilde S_3$.  The 
  elements are in a natural one-to-one correspondence with
  the triangles in the infinite Tonnetz.
  The left action of $\widetilde S_3$ on the Tonnetz gives rise to
  interesting chord sequences.
  We compare the system of transformations in $\widetilde S_3$ with the
  system of Schritte and Wechsel introduced by Hugo Riemann in 1880.
  Finally, we consider the point reflection group as it captures well
  the transition from Riemann's infinite Tonnetz to
  the finite Tonnetz of neo-Riemannian theory.

\end{minipage}
  \end{center}

\begin{abstract}
  The  PLR-moves of neo-Riemannian theory, when considered 
  as reflections on the edges of an equilateral triangle, define 
  the Coxeter group~$\widetilde S_3$.  The 
  elements are in a natural one-to-one correspondence with
  the triangles in the infinite Tonnetz.
  The left action of $\widetilde S_3$ on the Tonnetz gives rise to
  interesting chord sequences.
  We compare the system of transformations in $\widetilde S_3$ with the
  system of Schritte and Wechsel introduced by Hugo Riemann in 1880.
  Finally, we consider the point reflection group as it captures well
  the transition from Riemann's infinite Tonnetz to
  the finite Tonnetz of neo-Riemannian theory.
\end{abstract}

\section{PLR-moves revisited}

In neo-Riemannian theory, chord progressions are analyzed in terms
of elementary moves in the Tonnetz.
For example, the process of going from tonic to dominant
(which differs from the tonic chord
in two notes) is decomposed as a product
of two elementary moves of which each changes
only one note.

\smallskip
The three elementary moves considered are the PLR-transformations;
they map a major or minor triad to the minor or major triad
adjacent to one of the
edges of the triangle representing the chord in the Tonnetz.
See \cite{cohn} and \cite{cohn-audacious}.

\begin{figure}[ht]
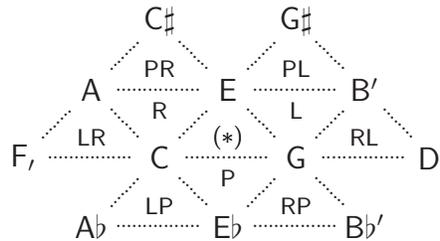

    \centerline{$    \beginpicture
    \setcoordinatesystem units <5ex,5ex>
    \put {\A\f} at 1 1 
    \put {\F\c} at 0 2
    \put {\E\f} at 3 1 
    \put {\C} at   2 2
    \put {\A} at   1 3 
    \put {\B\f\p} at 5 1 
    \put {\G} at 4 2 
    \put {\E} at 3 3
    \put {\C\s} at 2 4 
    \put {\D\p} at 6 2 
    \put {\B\p} at   5 3
    \put {\G\s} at   4 4 
    \put {\scriptsize $(*)$} at 3 2.3
    \setdots<2pt>
    \plot 0.4 2  1.6 2 /
    \plot 4.4 2  5.6 2 /
    \plot 1.4 1  2.6 1 /
    \plot 3.4 1  4.6 1 /
    \plot 1.4 3  2.6 3 /
     \plot 3.4 3  4.6 3 /
     \plot 2.4 2  3.6 2 /
     \plot 2.3 3.7  2.7 3.3 /
     \plot 4.3 1.7  4.7 1.3 /
     \plot 4.3 3.7  4.7 3.3 /
     \plot 5.3 2.7  5.7 2.3 /
     \plot 1.3 2.7  1.7  2.3 /
     \plot 2.3 1.7  2.7 1.3 /
     \plot 3.3 2.7  3.7 2.3 /
     \plot 0.3 2.3  0.7 2.7 /
     \plot 1.3 3.3  1.7 3.7 /
     \plot 1.3 1.3  1.7 1.7 /
     \plot 3.3 3.3  3.7 3.7 /
     \plot 2.3 2.3  2.7 2.7 /
     \plot 3.3 1.3  3.7 1.7 /
     \plot 4.3 2.3  4.7 2.7 /
     \put {\scriptsize $(*)$} at 3 2.3
     \put {\scriptsize\sf R} at 2 2.7
     \put {\scriptsize\sf L} at 4 2.7
     \put {\scriptsize\sf P} at 3 1.7
     \put {\scriptsize\sf LR} at 1 2.3
     \put {\scriptsize\sf PR} at 2 3.3
     \put {\scriptsize\sf LP} at 2 1.3
     \put {\scriptsize\sf PL} at 4 3.3
     \put {\scriptsize\sf RP} at 4 1.3
     \put {\scriptsize\sf RL} at 5 2.3
    \endpicture
    $}
    \caption{Triads in the vicinity of the \C-\E-\G-chord}
    \label{figure-PLR-2}
\end{figure}

Our paper is motivated by the observation that PLR-moves, while they provide
a tool to measure distance between triads, are not continuous as operations on the Tonnetz:

\smallskip
Let $s$ be a sequence of PLR-moves.
Applying $s$ to a pair of major chords results in a parallel shift
of those two chords.
However, applying the sequence $s$ to a major chord and an adjacent
minor chord makes the two chords drift apart.
For example, applying the sequence $s=\,${\scriptsize\sf RL} to the
chord labeled $(*)$ in Figure~\ref{figure-PLR-2}
yields the triad labeled {\scriptsize\sf RL}
on the right, while $s$ applied to {\scriptsize \sf P}
gives {\scriptsize\sf RLP} on the left, left of the triangle labelled {\scriptsize\sf LP}.

\smallskip
In this paper we consider 
the three reflections $s_1$, $s_2$, $s_3$ on the
edges of a fixed equilateral triangle, see Figure~\ref{figure-plr}.

\begin{figure}[ht]
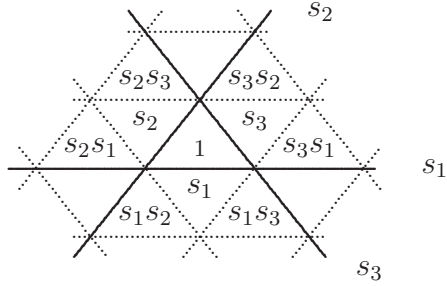

    \centerline{$    \beginpicture
      \setcoordinatesystem units <4ex,5ex>
      \put {} at -1 .5
      \put {} at 7 5
    \plot -.5 2  6.2 2 /
    \plot .7 .7   4.3 4.3 /
    \plot 1.7 4.3  5.3 .7 /
    \setdots<2pt>
    \plot .7 1  5.3 1 /
    \plot .7 3  5.3 3 /
    \plot 1.7 4  4.3 4 /
    \plot -.3 1.7  2.3 4.3 /
    \plot 2.7 .7  5.3 3.3 /
    \plot 4.7 .7  6.3 2.3 /
    \plot 3.7 4.3  6.3 1.7 /
    \plot -.3  2.3  1.3 .7 /
    \plot .7 3.3  3.3 .7 /
    \put {\scriptsize $1$} at 3 2.3
     \put {$s_1$} at 7.3 2
     \put {$s_2$} at 5.2 4.3
     \put {$s_3$} at 6.1 .5
     \put {$s_2$} at 2 2.7
     \put {$s_3$} at 4 2.7
     \put {$s_1$} at 3 1.7
     \put {$s_2s_1$} at 1 2.3
     \put {$s_2s_3$} at 2 3.3
     \put {$s_1s_2$} at 2 1.3
     \put {$s_3s_2$} at 4 3.3
     \put {$s_1s_3$} at 4 1.3
     \put {$s_3s_1$} at 5 2.3
    \endpicture
    $}
    \caption{The reflections $s_1$, $s_2$, $s_3$}
    \label{figure-plr}
\end{figure}

\smallskip
The three reflections satisfy the relations
$$s_1^2=s_2^2=s_3^2=1,\quad (s_1s_2)^3=(s_2s_3)^3=(s_3s_1)^3=1$$
which are the defining relations of the Coxeter group $\widetilde S_3$
corresponding to the affine irreducible Coxeter system
$\widetilde A_2$

\begin{figure}[h]
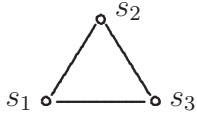

    \centerline{$    \beginpicture
      \setcoordinatesystem units <4ex,3ex>
      \circulararc 360 degrees from 0 .1 center at 0 0
      \circulararc 360 degrees from 1 2.1  center at 1 2
      \circulararc 360 degrees from 2 .1 center at 2 0
      \plot .2 0  1.8 0 /
      \plot .1 .15  .9 1.85 /
      \plot 1.9 .15  1.1 1.85 /
      \put {$s_1$} at -.5 0
      \put {$s_2$} at 1.5 2.2
      \put {$s_3$} at 2.5 0
    \endpicture
    $}
    \caption{Affine Coxeter system $\widetilde A_2$}
    \label{figure-a2}
\end{figure}

\section{The Coxeter group $\widetilde S_3$}
\label{section-coxeter}

We collect some results about the Coxeter group $\widetilde S_3$, most of the material
is adapted from \cite[Section~8.3]{bb}.
The group $\widetilde S_3$ can be realized as the group of \boldit{affine permutations},
$$\widetilde S_3=\{f:\mathbb Z\to \mathbb Z:f\,\text{bijective},f(-1)+f(0)+f(1)=0,\forall n:f(n+3)=f(n)+3\},$$
with multiplication the composition of maps.
Due to the last condition, it suffices to record the values of $f$ on the window $\{-1,0,1\}$, so
$$\widetilde S_3=\{[a,b,c]\in\mathbb Z^3:a+b+c=0, a,b,c\;\text{pairwise incongruent modulo}\;3\}.$$
Here, $1=[-1,0,1]$, $s_1=[0,-1,1]$, $s_2=[-1,1,0]$, $s_3=[-2,0,2]$ and composition of affine
permutations yields the multiplication rule
$$[a,b,c]\cdot s_i=\left\{\begin{array}{cl}\ [b,a,c], & \text{if}\;i=1\\
\ [a,c,b], & \text{if}\;i=2\\
\ [c-3,b,a+3], & \text{if}\;i=3.\end{array}\right.$$

\smallskip
Affine permutations, unlike sequences of reflections, provide a unique name for each element in $\widetilde S_3$.
The following result permits us to 
write the Coxeter element into the triangle to which
the corresponding sequence of reflections maps the identity element.
In Figure~\ref{figure-cox} we omit the brackets and place minus-signs under the numbers to improve readability.

\def\u{\underline}
\begin{figure}[ht]
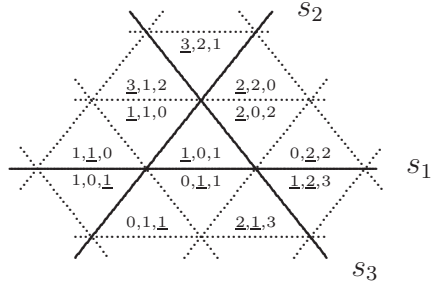

    \centerline{$    \beginpicture
      \setcoordinatesystem units <4ex,5ex>
      \put {} at -1 .5
      \put {} at 7 5
    \plot -.5 2  6.2 2 /
    \plot .7 .7   4.3 4.3 /
    \plot 1.7 4.3  5.3 .7 /
    \setdots<2pt>
    \plot .7 1  5.3 1 /
    \plot .7 3  5.3 3 /
    \plot 1.7 4  4.3 4 /
    \plot -.3 1.7  2.3 4.3 /
    \plot 2.7 .7  5.3 3.3 /
    \plot 4.7 .7  6.3 2.3 /
    \plot 3.7 4.3  6.3 1.7 /
    \plot -.3  2.3  1.3 .7 /
    \plot .7 3.3  3.3 .7 /
    \put {$\ssize \u1,0,1$} at 3 2.2
     \put {$s_1$} at 7 2
     \put {$s_2$} at 5 4.3
     \put {$s_3$} at 6 .5
     \put {$\ssize \u1,1,0$} at 2 2.8
     \put {$\ssize \u2,0,2$} at 4 2.8
     \put {$\ssize 0,\u1,1$} at 3 1.8
     \put {$\ssize 1,\u1,0$} at 1 2.2
     \put {$\ssize \u3,1,2$} at 2 3.2
     \put {$\ssize 0,1,\u1$} at 2 1.2
     \put {$\ssize \u2,2,0$} at 4 3.2
     \put {$\ssize \u2,\u1,3$} at 4 1.2
     \put {$\ssize 0,\u2,2$} at 5 2.2
     \put {$\ssize 1,0,\u1$} at 1 1.8
     \put {$\ssize \u3,2,1$} at 3 3.8
     \put {$\ssize \u1,\u2,3$} at 5 1.8
    \endpicture
    $}
    \caption{Affine permutations}
    \label{figure-cox}
\end{figure}

\begin{prop}\label{prop-correspondence}
  The elements in $\widetilde S_3$ are in one-to-one correspondence with the triangles in the Tonnetz.
\end{prop}

The result is well-known, in fact, the Tonnetz picture is commonly used to visualize the tesselation
of the affine plane given by the Coxeter system $\widetilde A_2$, see for example \cite[Figure~1.2]{bb}.
We give the proof to obtain relevant details of this tesselation.

\smallskip
Using the correspondence in Proposition~\ref{prop-correspondence} we will identify the elements in $\widetilde S_3$
with the triangles in the Tonnetz.  Thus, the group $\widetilde S_3$ acts on the Tonnetz via left multiplication
and via right multiplication, and both actions are simply transitive.

\begin{proof}
  The map given by sending a sequence of reflections to the triangle $\Delta$
  obtained by applying the reflections to the triangle marked $(*)$
  is an onto map:
  Unless $\Delta$ is the triangle $(*)$ itself, there exists at least one axis $s_i$
  between the two triangles.  Reflecting $\Delta$ on $s_i$ gives a triangle
  which is closer to $(*)$, hence the process of replacing $\Delta$ by $s_i(\Delta)$
  terminates after finitely many steps.

  \smallskip
  Equivalent sequences modulo the relations give the same triangle, hence we obtain a map
  from $\widetilde S_3$ to the set of triangles.  This map is injective since each triple
  $[a,b,c]$ records the coordinates of the triangle as described in the following lemma.\qed
\end{proof}

\begin{lem}
  For an affine permutation $[a,b,c]$ and for $i=1,\ldots,3$, define
  $$c_i([a,b,c])=\left\{\begin{array}{cl}\; a+1 \;& \text{if}\quad a\equiv i\mod 3\\
  b & \text{if}\quad b\equiv i\mod 3\\
  c-1 & \text{if}\quad c\equiv i\mod 3.  \end{array}\right.$$
  Then the coordinate of the center of the triangle corresponding to the Coxeter element $[a,b,c]$ in the Tonnetz
  with respect to the axis $s_i$ is $c_i$.
  (Here, the axis $s_1$ points to the left, $s_2$ to the top right and $s_3$
  to the bottom right.)
\end{lem}

For example, the three triangles in Figure~\ref{figure-cox} above and below the intersection of the $s_1$-
and the $s_3$-axis all have  $a=-2$.  Hence the $s_1$-coordinate of their center is $c_1=-1$.

\begin{proof}
  The numbers $c_i$ are defined since exactly one of the entries $a,b,c$ is
  congruent to $i$ modulo 3.
  The formula in the lemma can be verified using induction on the length of
  a sequence $s_{i_1}s_{i_2}\cdots s_{i_n}$ defining the Coxeter element
  and the above multiplication formula.\qed
\end{proof}

We have seen that for a given triangle $\Delta$, there is a unique Coxeter element which maps $(*)$ to $\Delta$.
The minimum number of reflections can be computed by counting inversions, see \cite[Proposition~8.3.1]{bb}
or by measuring the distance from $(*)$:

\begin{cor}\label{corollary}
  Suppose the triangle $\Delta$ corresponds to the affine permutation $[a,b,c]$.
  The minimum number $d$ of reflections needed to map $(*)$ to $\Delta$
  is the sum of the positive coordinates $c_i$, or $d=\frac 12\sum_{i=1}^3|c_i([a,b,c])|$.\qed
\end{cor}

As a product of reflections, each Coxeter element gives rise to an operation on the Tonnetz
which may be a translation, a rotation, a reflection or a glide reflection.
In Figure~\ref{figure-type}, we put a hook inside each triangle so that the type of operation given by an
element of $\widetilde S_3$ 
can be read off from the position of the hooks in two corresponding triangles.

\smallskip
For example, going from $(*)$ in Figure~\ref{figure-type} to (1) is a reflection, to (2) a rotation,
to (3) a glide reflection, and to (6) a translation.

\def\hru{\put {} at -2 -1
  \put {} at 2 1
  \plot -.45 -.75  .45 -.75  .3 -.6 / }
\def\hrd{\put {} at -2 -1
  \put {} at 2 1
  \plot -.45 .75  .45 .75  .3 .6 / }
\def\hdu{\put {} at -2 -1
  \put {} at 2 1
  \plot 1.31 -.19  .94 -.56  .74 -.36 / }
\def\hud{\put {} at -2 -1
  \put {} at 2 1
  \plot 1.31 .19  .94 .56  .74 .36 / }
\def\huu{\put {} at -2 -1
  \put {} at 2 1
  \plot -1 -.19  -1.31 -.19  -.94 -.56 / }
\def\hdd{\put {} at -2 -1
  \put {} at 2 1
  \plot -1 .19  -1.31 .19  -.94 .56 / }

\begin{figure}[ht]
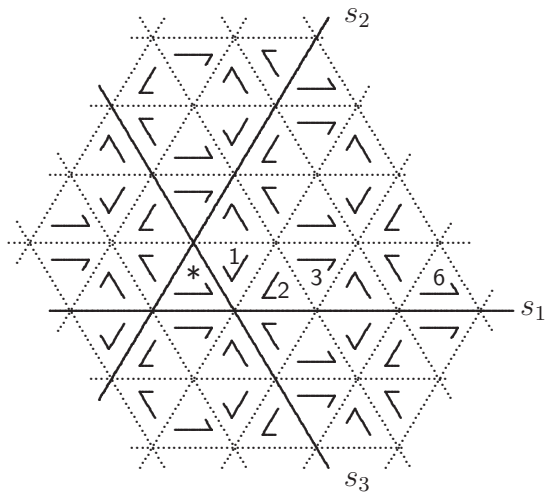

  \centerline{$    \beginpicture
    \setcoordinatesystem units <3ex,5ex>
      \put {} at -1 -1
      \put {} at 11 5.3
    \plot .5 1  11.8 1 /
    \plot 1.7 -.3   7.3 5.3 /
    \plot 1.7 4.3  7.3 -1.3 /
    \setdots<2pt>
    \plot -.5 2  10.8 2 /
    \plot .5 3  9.2 3 /
    \plot 1.5 4  8.2 4 /
    \plot 2.5 5  7.2 5 /
    \plot .5 1  11.2 1 /
    \plot 1.5 0  10.2 0 /
    \plot 2.5 -1  9.2 -1 /
    \plot .7 .7   5.3 5.3 /
    \plot -.3 1.7   3.3 5.3 /
    \plot 1.7 -.3   7.3 5.3 /
    \plot 2.7 -1.3   8.3 4.3 /
    \plot 4.7 -1.3   9.3 3.3 /
    \plot 6.7 -1.3   10.3 2.3 /
    \plot 8.7 -1.3   11.3 1.3 /
    \plot -.3 2.3  3.3 -1.3 /
    \plot .7 3.3  5.3 -1.3 /
    \plot 2.7 5.3  9.3 -1.3 /
    \plot 4.7 5.3  10.3 -.3 /
    \plot 6.7 5.3  11.3 .7 /
    \setsolid
    \multiput {\hru \hrd \hdu \hud \huu \hdd } at  4 4  4 2  4 0  7 3  7 1  /
    \put {\hdd \hrd \hud } at 7 -1
    \put {\hru \hdu \hud } at 1 3
    \put {\hrd \hud \hdu } at 1 1 
    \put {\huu \hru } at 7 5
    \put {\hdd \huu \hru } at 10 2
    \put {\huu \hdd \hrd } at 10 0
     \put {$s_1$} at 12.3 1
     \put {$s_2$} at 8 5.3
     \put {$s_3$} at 8 -1.5
     \put {\sf *} at 4 1.5
     \put {\scriptsize\sf 1} at 5 1.8
     \put {\scriptsize\sf 2} at 6.2 1.3
     \put {\scriptsize\sf 3} at 7 1.5
     \put {\scriptsize\sf 6} at 10 1.5
     %
    \endpicture
    $}
    \caption{Left multiplication by affine permutations}
    \label{figure-type}
\end{figure}

\section{The fundamental hexagon}

It turns out that the Coxeter group $\widetilde S_3$ has a normal subgroup $\widetilde T$ of index 6
given by translations.  The factor group $\widetilde S_3/\widetilde T$ is isomorphic
to the symmetric group $S_3$.
We call the fundamental domain with respect to the shift by a translation in $\widetilde T$
the fundamental hexagon.  In the next section, we will discuss the role of this fundamental hexagon
in music.

\smallskip
It would be desirable to have a ``comma subgroup'' in $\widetilde S_3$ to provide a link to the
group generated by the PLR-moves on the finite Tonnetz, but the author was not able to detect a
suitable normal subgroup in $\widetilde S_3$.
However, there is a related group, the point reflection group $P$,
which does have such a ``comma subgroup'', as we will see in Section~\ref{section-point}.

\smallskip
Note that left multiplication by $s_2s_3s_2=s_3s_2s_3$ is the reflection on the Tonnetz
on the line one unit above the $s_1$-axis.  Hence $t_1=(s_2s_3s_2)s_1$ is the upwards translation
by 2 units.
Similarly, $t_2=(s_3s_1s_3)s_2$ and $t_3=(s_1s_2s_1)s_3$ are translations by 2 units
towards the lower right and the lower left, respectively, as indicated in Figure~\ref{figure-hexagons}.

\begin{prop}
  The group of translations $\widetilde T=\langle t_1,t_2,t_3\rangle=\langle t_1,t_2\rangle$
  is a free abelian group of rank 2.
  Moreover, $\widetilde T$ is a normal subgroup of $\widetilde S_3$ of index 6,
  the factor group $\widetilde S_3/\widetilde T$ is isomorphic to the symmetric group $S_3$.
\end{prop}

\begin{proof}
  From Euclidean geometry it is clear that $t_1t_2=t_3^{-1}=t_2t_1$, it follows that
  $\widetilde T=\langle t_1,t_2\rangle$ is an abelian group; moreover, the translations
  $t_1,t_2$ span a 2-dimensional lattice in the plane.
  The group $\widetilde T$ is a normal subgroup of $\widetilde S_3$ since
  $s_1t_1s_1^{-1}=t_1^{-1}$, $s_2t_1s_2^{-1}=t_3^{-1}$, $s_3t_1s_3^{-1}=t_2^{-1}$,
  and the index is 6 since the plane of the Tonnetz can be tiled with hexagons
  which are in one-to-one correspondence with the elements in $\widetilde T$.

  \def\hexa{\setsolid
    \plot -2 0  -1 1  1 1  2 0  1 -1  -1 -1  -2 0 /
    \setdots<2pt>
    \plot -2 0  2 0 /
    \plot -1 1  1 -1 /
    \plot -1 -1  1 1 /  }
  
  \begin{figure}[ht]
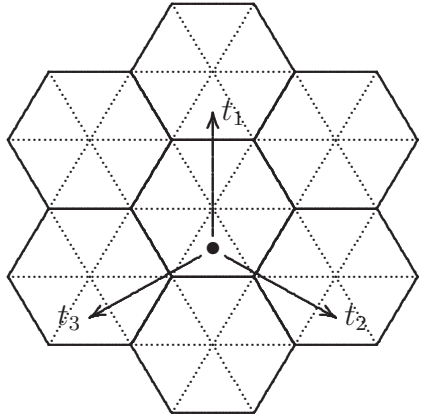

    \centerline{$    \beginpicture
      \setcoordinatesystem units <3ex,5ex>
      \put {} at -4 -3 
      \put {} at 4 3 
      \setsolid
      \put{$\bullet$} at 0 -.6
      \arr{0 -.4}{0 1.4}
      \arr{.3 -.7}{3 -1.6}
      \arr{-.3 -.7}{-3 -1.6}
      \put{$t_1$} at .5 1.4
      \put{$t_2$} at 3.5 -1.6
      \put{$t_3$} at -3.5 -1.6
      \multiput{\hexa} at 0 0  3 1  0 2  -3 1  -3 -1  0 -2  3 -1 /
      \endpicture
      $}
    \caption{Translations of hexagons}
    \label{figure-hexagons}
  \end{figure}
  For the last claim, consider $S_3=\langle s_2, s_3\rangle$ as a subgroup of $\widetilde S_3$.
  Since $S_3\cap \widetilde T=\{e\}$, the composition
  $$S_3\to \widetilde S_3\to \widetilde S_3/\widetilde T$$
  is a one-to-one map, hence a group isomorphism.\qed
\end{proof}

\begin{cor}
  As a group, $\widetilde S_3=\widetilde T\cdot S_3$.\qed
\end{cor}

We call the region in the Tonnetz corresponding to the subgroup $S_3=\langle s_2,s_3\rangle$
the \boldit{fundamental hexagon} (see Proposition~\ref{prop-correspondence}).
The left cosets $tS_3$ for $t\in\widetilde T$ form the tiling pictured in Figure~\ref{figure-hexagons}.
Right multiplication by an element in $S_3$ yields a permutation of the triangles within
each hexagon, while left multiplication by an element of $\widetilde T$ is a parallel shift
which preserves the hexagonal pattern.

\smallskip
We conclude this section by briefly lising the four types of
elements in $\widetilde S_3$ in terms of their action on the Tonnetz given by
left multiplication.

\begin{itemize}
\item The elements of order 2 are reflections on a line parallel to one of the axes.
  In particular, the reflection on the $n$-th line parallel to $s_i$ is given by $t_i^ns_i$
  ($n\in\mathbb Z, i=1,2,3$).
\item Products of two reflections on lines which are not parallel
  are rotations by $\pm120^\circ$ about a vertex in the Tonnetz.
  Those are the elements of order 3 in $\widetilde S_3$.
\item Products of two reflections on parallel lines are translations,
  they form the normal subgroup $\widetilde T$ considered above.
\item The remaining elements are odd, they act as glide reflections on the Tonnetz;
  all have infinite order.
\end{itemize}

\section{The fundamental hexagon in music}

The hexagon encapsulates fundamental concepts in music theory, and leads to some, perhaps weird,
sequences of chords.

\smallskip
Each hexagon in the Tonnetz consists of six triangles which represent major and minor chords
which have one note in common, see Figure~\ref{figure-chords-in-hexagon}.

  \begin{figure}[ht]
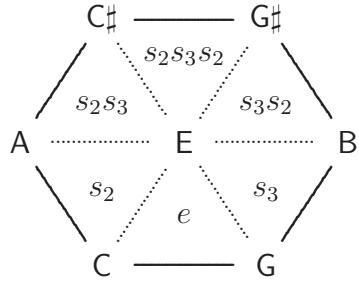

    \centerline{$    \beginpicture
      \setcoordinatesystem units <6ex,9ex>
      \put {} at -2 -1
      \put {} at 2 1
      \setsolid
      \put{\C\s} at -1 1
      \put{\G\s} at 1 1
      \put{\A} at -2 0
      \put{\E} at 0 0
      \put{\B} at 2 0
      \put{\C} at -1 -1
      \put{\G} at 1 -1
      \put{$s_2$} at -1 -.4
      \put{$s_3$} at 1 -.4
      \put{$s_3s_2$} at 1 .3
      \put{$s_2s_3$} at -1 .3
      \put{$s_2s_3s_2$} at 0 .7
      \put{$e$} at 0 -.6
      \plot -.6 1  .6 1 /
      \plot -1.8 .2  -1.2 .8 /
      \plot -1.8 -.2  -1.2 -.8 /
      \plot -.6 -1  .6 -1 /
      \plot 1.8 .2  1.2 .8 /
      \plot 1.8 -.2  1.2 -.8 /
      \setdots<2pt>
      \plot -1.6 0  -.4 0 /
      \plot .4 0  1.6 0 /
      \plot -.8 -.8  -.2 -.2 /
      \plot .2 .2  .8 .8 /
      \plot -.8 .8  -.2 .2 /
      \plot .2 -.2  .8 -.8 /
      \endpicture
      $}
    \caption{Six major and minor chords}
    \label{figure-chords-in-hexagon}
  \end{figure}

  \smallskip
  Consider the \E-hexagon from Figure~\ref{figure-chords-in-hexagon}.
  The three reflections $s_2, s_3, s_2s_3s_2$ in $S_3$ describe the
  PLR-moves locally.

  \smallskip
  Reflection on the $s_3$-axis is the leading tone exchange:
  $$\LL: \quad \text{\C-\E-\G} \longleftrightarrow \text{\E-\G-\B},
  \quad \text{\A-\C\s-\E} \longleftrightarrow \text{\C\s-\E-\G\s}$$
  Reflection on the $s_2$-axis yields the relative major or minor:
  $$\RR: \quad \text{\C-\E-\G}\longleftrightarrow \text{\A-\C-\E},
  \quad \text{\E-\G\s-\B}\longleftrightarrow\text{\C\s-\E-\G\s}$$
  and reflection on the $s_2s_3s_2$-axis the parallel major or minor:
  $$\PP:  \quad \text{\A-\C\s-\E} \longleftrightarrow \text{\A-\C-\E},
  \quad\text{\E-\G\s-\B}\longleftrightarrow\text{\E-\G-\B}.$$

  \smallskip
  The elements $s_2s_3$ and $s_3s_2$ in $S_3$ have order 3.  Iterated multiplication by $s_3s_2$
  gives rise to a 3-cycle of major chords within the \E-hexagon:
  $$\text{\C-\E-\G}\longrightarrow\text{\E-\G\s-\B}\longrightarrow\text{\A-\C\s-\E}\longrightarrow\text{\C-\E-\G}$$
  and a 3-cycle of minor chords:
  $$\text{\C\s-\E-\G\s}\longrightarrow\text{\A-\C-\E}\longrightarrow\text{\E-\G-\B}\longrightarrow\text{\C\s-\E-\G\s}.$$

  The three steps in the cycle:  The move to the upper left, then the horizontal move to the right and the move to the lower left,
  mark three stripes in the Tonnetz.
  \begin{itemize}
  \item The horizontal stripe given by a chord contains all possibly higher subdominant
    and dominant chords.
  \item The stripe in NE-SW direction pictures the (infinite) hexatonic
    system to which the chord belongs,
    see \cite[Part~III]{cohn96}.
  \item The stripe in NW-SE direction represents the
    (infinite) octatonic system for the given chord, see \cite{cohn97}.
  \end{itemize}
  We notice that while the rotations in $S_3$ mark the directions of the three stripes, the translations in $\widetilde T$
  can be used to move between parallel systems.

  \smallskip\sloppypar
  We would like to point out that the three opening chords of Ludwig van Beethoven's {\it Moonlight Sonata}
  take place within the \E-hexagon.  The \C\s-minor chord leads to the \C\s-minor sept chord \C\s-\E-\G\s-\B\ 
  (which contains the relative \E-major chord), then to the (subdominant) \A-major chord.
  The neighboring \A-hexagon captures the transition from the \A-major chord to the following
  subdominant \D-major chord...

  \smallskip
  There are more, perhaps even weirder, chord sequences outside of
  the central hexagon.
  For examle, the rotation by $s_3s_2$ permutes the major chords which have one edge in common with the
  \E-hexagon:
  $$\text{\G-\B-\D}\longrightarrow\text{\C\s-\E\s-\G\s}\longrightarrow\text{\F-\A-\C}\longrightarrow\text{\G-\B-\D}$$
  and similarly the minor chords:
  $$\text{\C-\E\f-\G}\longrightarrow\text{\G\s-\B-\D\s}\longrightarrow\text{\F\s-\A-\C\s}\longrightarrow\text{\C-\E\f-\G}.$$

  \smallskip
  Another type of chord sequence is obtained from translations of the hexagons in the Tonnetz.
  Consider the tiling pictured in Figure~\ref{figure-tiling-by-hexagons}.
  
    \def\hexaempty{\setsolid
    \plot -2 0  -1 1  1 1  2 0  1 -1  -1 -1  -2 0 /
    \setdots<2pt>
    \plot -2 0 -.6 0 /
    \plot .6 0  2 0 /
    \plot -1 1  -.3 .3 /
    \plot .3 -.3  1 -1 /
    \plot -1 -1  -.3 -.3 /
    \plot .3 .3  1 1 /  }
  
  \begin{figure}[ht]
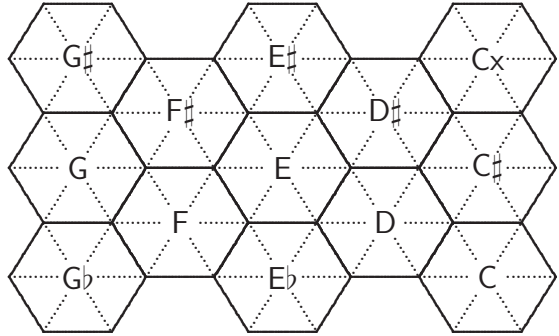

    \centerline{$    \beginpicture
      \setcoordinatesystem units <2.5ex,4ex>
      \put {} at -4 -3 
      \put {} at 4 3 
      \multiput{\hexaempty} at 0 0  3 1  0 2  -3 1  -3 -1  0 -2  3 -1  6 2  6 0  6 -2  -6 2  -6 0  -6 -2 /
      \put{\E} at 0 0
      \put{\E\s} at 0 2
      \put{\E\f} at 0 -2
      \put{\F} at -3 -1
      \put{\F\s} at -3 1
      \put{\D} at 3 -1
      \put{\D\s} at 3 1
      \put{\G} at -6 0
      \put{\G\s} at -6 2
      \put{\G\f} at -6 -2
      \put{\C\s} at 6 0
      \put{\C{\sf x}} at 6 2
      \put{\C} at 6 -2
      \endpicture
      $}
    \caption{The tiling by hexagons}
    \label{figure-tiling-by-hexagons}
  \end{figure}

  \smallskip
  The translations which define the tiling 
  satisfy the identity $t_3t_2t_1=e$.  Applying successively $t_1, t_2t_1, t_3t_2t_1$ to \C-\E-\G\
  yields the sequence of major chords
  $$\text{\C-\E-\G}\longrightarrow \text{\C\s-\E\s-\G\s}\longrightarrow\text{\B-\D\s-\F\s}\longrightarrow\text{\C-\E-\G}$$
  (in the \E-, \E\s-, and \D\s-hexagons), and the sequence of minor chords
  $$\text{\A-\C-\E}\longrightarrow \text{\A\s-\C\s-\E\s}\longrightarrow\text{\G\s-\B-\D\s}\longrightarrow\text{\A-\C-\E}$$
  (also in the \E-, \E\s-, and \D\s-hexagons).

  \smallskip
  A substantial collection of Tonnetz models can be found in \cite{cohn-audacious}.
  An interesting musical case for the application of $\widetilde S_3$
  are the ``pitch retention loops'' in which the chords in the hexagon
  occur in cyclical order, see in particular \cite[Figure~6.3]{cohn-audacious}.
  The Coxeter group $\widetilde S_3$ does not contain any elements of order 6,
  so the rotation by $60^\circ$ may be difficult to explain.
  But the alteration of $s_2$ and $s_3$ is still a more effective
  description than the succession of $L$, $P$ and $R$.

  \smallskip
  In the above, we have identified the infinite triadic Tonnetz
  with the Coxeter group $\widetilde S_3$ and studied the left
  action of the group $\widetilde S_3$ on itself.
  Considering the right action, note that alternating right
  multiplication by $s_3$ and $s_2$ generates the triads in each
  of the hexagons with base triad in $\widetilde T$,
  they are pictured in Figure~\ref{figure-tiling-by-hexagons}.
  The remaining hexagons in the Tonnetz have their base triad
  in either $s_3s_2\widetilde T$ or $s_2s_3\widetilde T$,
  there the triads in the cycle are generated by alternating
  right multiplication by $s_1$ and $s_3$, or by $s_2$ and $s_1$,
  respectively.

\section{Schritte and Wechsel, revisited}

In \cite{rie}, Hugo Riemann presents two kinds of operations
on the Tonnetz, Schritte and Wechsel.  Under a \boldit{Schritt,} each note in a major triad moves up or down
a certain number of scale degrees, while the notes in a minor triad move in the opposite direction.
For example, using neo-Riemannian terminology, the Quintschritt is given by the
\RR\LL-move, the Terzschritt by the \PP\LL-move.
Under a \boldit{Wechsel,} major and minor triads correspond to each other,
for example the Seitenwechsel $w$ yields the parallel triad given by the \PP-move.

\begin{figure}[ht]
  \centerline{$    \beginpicture
    \setcoordinatesystem units <3ex,5ex>
    \put {} at -.3 -.3
      \put {} at 6.3 3.3 
      \setdots<2pt>
      \plot .7 0  5.3 0 /
      \plot -.3 1  6.3 1 /
      \plot .7 2  5.3 2 /
      \plot 1.7 3  4.3 3 /
      \plot -.2 .8  2.2 3.2 /
      \plot .8 -.2  4.2 3.2 /
      \plot 2.8 -.2  5.2 2.2 /
      \plot 4.8 -.2  6.2 1.2 /
      \plot -.2 1.2  1.2 -.2 /
      \plot .8 2.2  3.2 -.2 /
      \plot 1.8 3.2  5.2 -.2 /
      \plot 3.8 3.2  6.2 .8 /
      \setsolid
      \arr {2.3 2.3} {3.7 2.3}
      \arr {1.3 1.3} {2.7 1.3}
      \arr {3.3 1.3} {4.7 1.3}
      \arr {2.3 .3} {3.7 .3}
      \arr {2.7 .7} {1.3 .7}
      \arr {4.7 .7} {3.3 .7}
      \arr {3.7 1.7} {2.3 1.7}
      \multiput{$\scriptstyle\bullet$} at 2.3 2.3  3.3 1.3  1.3 1.3  2.3 .3  2.7 .7  4.7 .7 3.7 1.7 /
      \endpicture
      \qquad
       \beginpicture
    \setcoordinatesystem units <3ex,5ex>
      \put {} at -.3 -.3
      \put {} at 6.3 3.3 
      \setdots<2pt>
      \plot .7 0  5.3 0 /
      \plot -.3 1  6.3 1 /
      \plot .7 2  5.3 2 /
      \plot 1.7 3  4.3 3 /
      \plot -.2 .8  2.2 3.2 /
      \plot .8 -.2  4.2 3.2 /
      \plot 2.8 -.2  5.2 2.2 /
      \plot 4.8 -.2  6.2 1.2 /
      \plot -.2 1.2  1.2 -.2 /
      \plot .8 2.2  3.2 -.2 /
      \plot 1.8 3.2  5.2 -.2 /
      \plot 3.8 3.2  6.2 .8 /
      \setsolid
      \arr {1 1.5} {1.7 2.2}
      \arr {3 1.5} {3.7 2.2}
      \arr {2 .5} {2.7 1.2}
      \arr {4 .5} {4.7 1.2}
      \arr {3 2.5} {2.3 1.8}
      \arr {2 1.5} {1.3 .8}
      \arr {4 1.5} {3.3 .8}
      \multiput{$\scriptstyle\bullet$} at 1 1.5  3 1.5  2 .5  4 .5  3 2.5  2 1.5  4 1.5 /
      \endpicture\qquad
       \beginpicture
    \setcoordinatesystem units <3ex,5ex>
      \put {} at -.3 -.3
      \put {} at 6.3 3.3 
      \setdots<2pt>
      \plot .7 0  5.3 0 /
      \plot -.3 1  6.3 1 /
      \plot .7 2  5.3 2 /
      \plot 1.7 3  4.3 3 /
      \plot -.2 .8  2.2 3.2 /
      \plot .8 -.2  4.2 3.2 /
      \plot 2.8 -.2  5.2 2.2 /
      \plot 4.8 -.2  6.2 1.2 /
      \plot -.2 1.2  1.2 -.2 /
      \plot .8 2.2  3.2 -.2 /
      \plot 1.8 3.2  5.2 -.2 /
      \plot 3.8 3.2  6.2 .8 /
      \setsolid
      \arr {1 1} {1 .5}
      \arr {1 1} {1 1.5}
      \arr {3 1} {3 .5}
      \arr {3 1} {3 1.5}
      \arr {5 1} {5 .5}
      \arr {5 1} {5 1.5}
      \arr {2 2} {2 1.5}
      \arr {2 2} {2 2.5}
      \arr {4 2} {4 1.5}
      \arr {4 2} {4 2.5}
      \endpicture
    $}
    \caption{Quintschritt \RR\LL, Terzschritt \PP\LL, and Seitenwechsel \PP}
    \label{figure-schritt-wechsel}
\end{figure}

\smallskip
For a detailed description of Riemann's system of Schritte and Wechsel we refer to \cite{klum}
which also discusses the composition of Schritte and Wechsel.
The Schritte group $T$ is generated by the Quintschritt \RR\LL\ and the Terzschritt \PP\LL,
it is isomorphic to the additive group $\mathbb Z\times\mathbb Z$.
Note that under a Schritt,
major and minor triads move in opposite directions, see Figure~\ref{figure-schritt-wechsel}.

\smallskip
Each Wechsel can be obtained as a composition of a Schritt with the Seitenwechsel $w$,
hence the product of a Wechsel with itself is the identity operation on the Tonnetz.
More generally, if $t,t'\in T$ are Schritte, 
then then the product $(t'w)\cdot(tw)=t't^{-1}$ of two Wechsel is the
composition of a Schritt $t'$ with the opposite of the Schritt $t$ --- hence the product of two Wechsel is always
a Schritt.
The map $\varphi:T\to T, t\mapsto t^{-1}$ given by conjugation by $w$ is an
automorphism of order two (since $T$ is an abelian group), and the group $R$ of Schritte and Wechsel as described in
\cite[Appendix III]{klum} is the semi-direct product $\mathbb Z_2\ltimes_\varphi T$.

\smallskip
We compare the Schritt-Wechsel group $R$ and the Coxeter group $\widetilde S_3$.

\smallskip
As for the Coxeter group $\widetilde S_3$,
the elements in $R$ are in one-to-one correspondence with the triangles in the infinite Tonnetz.
Using this identification, the Schritt-Wechsel group $R$ acts on itself via left multiplication
and via right multiplication; each action is simply transitive.

\smallskip
There is a normal subgroup $N$ in $\widetilde S_3$ of index 2, it is given by all sequences of
reflections of even length.  Under the identification of $\widetilde S_3$ with the
triangles in the Tonnetz, the subgroup $N$ corresponds to the triangles of shape $\Delta$,
which are the triangles in even distance from $(*)$ (see Corollary~\ref{corollary}).
As a subset of $\widetilde S_3$, the group $N$ consists of the rotations and the translations.

\smallskip
For a reflection in $\widetilde S_3$, say $s_1$, 
the conjugation by $s_1$ defines a group automorphism $\psi:N\to N, n\mapsto s_1ns_1^{-1}$.
The group $\widetilde S_3$ is isomorphic to the semi-direct product
$\mathbb Z_2\ltimes_\psi N$.

\smallskip
Despite these structural similarities, the groups $R$, $\widetilde S_3$ are not isomorphic.

\section{The point reflection group}
\label{section-point}

To shed light on the interplay between Riemannian theory on the infinite Tonnetz and
neo-Riemannian theory on the finite Tonnetz, we exhibit a third group (besides $\widetilde S_3$ and $R$),
the point reflection group $P$.  It has three main features:
\begin{itemize}
\item[$\bullet$] The group $P$ is generated by three reflections (that is, elements of order 2).
\item[$\bullet$] The group $P$ is naturally isomorphic to the opposite group of $R$.
\item[$\bullet$] There is a normal subgroup $K$ in $P$ with factor the dihedral group $D_{12}$.
\end{itemize}

Recall that $D_{12}$ is the group of 24 elements, generated by the PLR-moves on the
finite Tonnetz, see \cite[Chapter~5]{cfs}.

\smallskip
The \boldit{point reflection group} $P$ is the subgroup of the group of Euclidean
plane isometries
generated by the $180^\circ$ rotations $\pi_1,\pi_2,\pi_3$, where each $\pi_i$ is the point reflection
about the midpoint of the edge of (*) on the $s_i$-axis, see Figure~\ref{figure-point-reflections}.

\begin{figure}[ht]
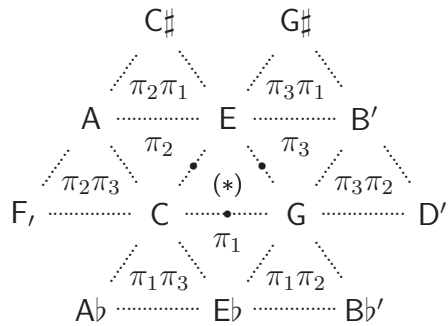

    \centerline{$    \beginpicture
    \setcoordinatesystem units <5ex,7ex>
    \put {\A\f} at 1 1 
    \put {\F\c} at 0 2
    \put {\E\f} at 3 1 
    \put {\C} at   2 2
    \put {\A} at   1 3 
    \put {\B\f\p} at 5 1 
    \put {\G} at 4 2 
    \put {\E} at 3 3
    \put {\C\s} at 2 4 
    \put {\D\p} at 6 2 
    \put {\B\p} at   5 3
    \put {\G\s} at   4 4 
    \put {\scriptsize $(*)$} at 3 2.3
    \setdots<2pt>
    \plot 0.4 2  1.6 2 /
    \plot 4.4 2  5.6 2 /
    \plot 1.4 1  2.6 1 /
    \plot 3.4 1  4.6 1 /
    \plot 1.4 3  2.6 3 /
     \plot 3.4 3  4.6 3 /
     \plot 2.4 2  3.6 2 /
     \plot 2.3 3.7  2.7 3.3 /
     \plot 4.3 1.7  4.7 1.3 /
     \plot 4.3 3.7  4.7 3.3 /
     \plot 5.3 2.7  5.7 2.3 /
     \plot 1.3 2.7  1.7  2.3 /
     \plot 2.3 1.7  2.7 1.3 /
     \plot 3.3 2.7  3.7 2.3 /
     \plot 0.3 2.3  0.7 2.7 /
     \plot 1.3 3.3  1.7 3.7 /
     \plot 1.3 1.3  1.7 1.7 /
     \plot 3.3 3.3  3.7 3.7 /
     \plot 2.3 2.3  2.7 2.7 /
     \plot 3.3 1.3  3.7 1.7 /
     \plot 4.3 2.3  4.7 2.7 /
     \put {$\ssize\bullet$} at 3 2 
     \put {$\ssize\bullet$} at 2.5 2.5
     \put {$\ssize\bullet$} at 3.5 2.5
     \put {\scriptsize $(*)$} at 3 2.3
     \put {$\pi_2$} at 2 2.7
     \put {$\pi_3$} at 4 2.7
     \put {$\pi_1$} at 3 1.7
     \put {$\pi_2\pi_3$} at 1 2.3
     \put {$\pi_2\pi_1$} at 2 3.3
     \put {$\pi_1\pi_3$} at 2 1.3
     \put {$\pi_3\pi_1$} at 4 3.3
     \put {$\pi_1\pi_2$} at 4 1.3
     \put {$\pi_3\pi_2$} at 5 2.3
    \endpicture
    $}
    \caption{Point reflections on the Tonnetz}
    \label{figure-point-reflections}
\end{figure}

\smallskip
The action on the Tonnetz given by left multiplication by point reflections is pictured in
Figure~\ref{figure-Paction}.  By comparison, Figure~\ref{figure-type} in Section~\ref{section-coxeter}
shows the action by affine permutations.

\def\hleft{\put {} at -2 -1
  \put {} at 2 1
  \plot -.3 -.3  -.45 -.15  .45 -.15  / }
\def\hright{\put {} at -2 -1
  \put {} at 2 1
  \plot -.45 .15  .45 .15  .3 .3 / }

\begin{figure}[ht]
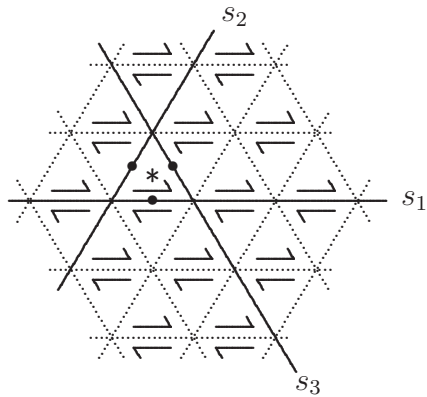

  \centerline{$    \beginpicture
    \setcoordinatesystem units <3ex,5ex>
      \put {} at -.3 -.5
      \put {} at 9 4.3
    \plot -.5 2  8.7 2 /
    \plot .7 .7   4.5 4.5 /
    \plot 1.7 4.3  6.5 -.5 /
    \setdots<2pt>
    \plot .5 3  7.4 3 /
    \plot 1.5 4  6.4 4 /
    \plot .5 1  7.4 1 /
    \plot 1.5 0  6.4 0 /
    \plot -.3 1.7   2.3 4.3 /
    \plot 1.7 -.3   6.3 4.3 /
    \plot 3.7 -.3   7.3 3.3 /
    \plot 5.7 -.3   8.3 2.3 /
    \plot -.3 2.3  2.3 -.3 /
    \plot .7 3.3  4.3 -.3 /
    \plot 3.7 4.3  7.3 .7 /
    \plot 5.7 4.3  8.3 1.7 /
    %
    \setsolid
    \multiput {\hleft \hright } at 1 2  2 1  2 3  3 0  3 2  3 4  4 1  4 3
                5 0  5 2  5 4  6 1  6 3  7 2   /
     \put {$s_1$} at 9.4 2
     \put {$s_2$} at 5 4.7
     \put {$s_3$} at 6.8 -.7
     \put {\sf *} at 3 2.3
     \multiput {\scriptsize$\bullet$} at  3 2  2.5 2.5  3.5 2.5 / 
    \endpicture
    $}
    \caption{Left multiplication by point reflections}
    \label{figure-Paction}
\end{figure}

\smallskip
The product of two point reflections is the translation by twice the difference between the centers,
so for example, $\pi_3\pi_2$ is the shift by one unit to the right in in parallel to the $s_1$-axis.
Hence the collection of all products of an even number of point reflections
forms the group of translations $T$.

\smallskip
Each remaining element in $P$ is a product of a point reflection and a translation,
hence a point reflection itself, about a vertex or about the midpoint of a triangle edge in the
Tonnetz.  Each such element has order 2.

\begin{prop}\label{prop-RPiso}
  The point reflection group $P$ is in a natural way isomorphic to
  the opposite group of Riemann's Schritte-Wechsel group $R$.
\end{prop}

\begin{proof}
  Identify the elements of $P$ with the triads in the Tonnetz, as indicated in
  Figure~\ref{figure-Paction}.  Then right multiplication by $\pi_3\pi_2$ is the
  Quintschritt \RR\LL, by $\pi_3\pi_1$ the Terzschritt \PP\LL, and by $\pi_1$ the Seitenwechsel
  \PP, see Figure~\ref{figure-point-reflections}.
  Quintschritt, Terzschritt and Seitenwechsel generate the group $R$,
  and the elements $\pi_3\pi_2$, $\pi_3\pi_1$ and $\pi_1$ generate $P$.
  Hence the left action of $R$ on the Tonnetz coincides with the right action
  of $P$.
  Both actions are simply transitive; it follows that the groups $R^{\rm op}$ and $P$ are isomorphic.\qed
\end{proof}

It seems to be well known that Riemann's Schritt-Wechsel group $R$ has the Comma-Schritte subgroup as a normal
subgroup such that the Schritt-Wechsel group of the finite Tonnetz, which is isomorphic to the dihedral group
$D_{12}$, is a factor.  The corresponding result for the (isomorphic) point reflection group can be obtained directly.

\begin{defin}
  The \boldit{comma subgroup} $K$ is the subgroup of the point reflection group $P$
  generated by $(\pi_3\pi_1)^3$ and $(\pi_1\pi_2)^4$.
\end{defin}

One can see that $K$ is the smallest subgroup of $P$ which contains the
``Lesser-Diesis-Schritt'' $(\pi_3\pi_1)^3$, the ``Greater-Diesis-Schritt'' $(\pi_1\pi_2)^4$,
the ``Syntonic-Comma-Schritt'' $(\pi_3\pi_2)^3\pi_1\pi_2$, and the ``Pythagorean-Comma-Schritt''
$(\pi_3\pi_2)^{12}$. 

\begin{prop}
  The comma subgroup $K$ is a normal subgroup of $P$ of index 24.
  The factor group $P/K$ is isomorphic to the dihedral group $D_{12}$.
\end{prop}

\begin{proof}
  \sloppypar
  Using the formula from the proof of Proposition~\ref{prop-RPiso}, we obtain
  $\pi_i(\pi_3\pi_1)^3\pi_i^{-1}=(\pi_1\pi_3)^3=((\pi_3\pi_1)^3)^{-1}\in K$
  and similarly, $\pi_i(\pi_1\pi_2)^4\pi_i^{-1}\in K$, so $K$ is normal in $P$.
  The map $\psi:\mathbb Z\times\mathbb Z\to P, (a,b)\mapsto (\pi_3\pi_1)^a(\pi_1\pi_2)^b$
  is one-to-one and has image $T$.  The subgroup generated by $(3,0)$ and $(0,4)$
  of $\mathbb Z\times\mathbb Z$ corresponds to $K$ under $\psi$.  Hence $K$ has index 12
  in $T$ and index 24 in $P$.

  \smallskip
  Let $h=(\pi_1\pi_2)^{-1}(\pi_3\pi_1)$ (the semitone) and $\rho=\pi_1$.  Then $hK$ has order 12 in $P/K$,
  $\rho K$ has order 2, and $(\rho K)(hK)(\rho K)^{-1}=(hK)^{-1}$.
  Thus, $hK$ and $\rho K$ generate a subgroup in $P/K$ isomorphic to $D_{12}$.\qed
\end{proof}

\section{Conclusion}

In Riemannian theory, the vertices in the infinite Tonnetz are labeled by notes such that the fifth
marks the horizontal direction.  All operations on the Tonnetz perserve the horizontal direction:
The Schritte are translations, and the Wechsel are products of a translation and a flip on the
horizontal axis.

\smallskip
Neo-Riemannian theory purifies the Tonnetz by removing the labels attached to the vertices, and by identifying
the triangles with chords.  This allows to redefine the operations in terms of more basic
reflections, which in turn give rise to new moves, in particular to rotations.

\smallskip
The group which defines the operations in Riemannian theory is the semi-direct product
$R=\mathbb Z_2\ltimes_\varphi T$ of the cyclic group of two elements by the group of
affine translations in the plane.  

\smallskip
By comparison, three reflections corresponding to PLR-moves generate 
the Coxeter group $\widetilde S_3$ which acts on the infinite Tonnetz.
Like $R$, the group $\widetilde S_3$ contains a subgroup, say $\widetilde T$, of translations;
actually $\widetilde T$ is isomorphic to $T$.  Unlike $R$, the translation subgroup has index six,
so there are many more elements in $\widetilde S_3$: reflections, $120^\circ$-rotations
and glide reflections. 

\smallskip
What has changed since Hugo Riemann introduced Schritte and Wechsel?  We still visualize music in the Tonnetz...
We still use algebra to describe the development of harmony...  Yet, the buiding blocks are more fundamental
and the operations have more variety.  Riemannian theory is very much alive.

\bigskip
Happy 170th Birthday, Hugo Riemann!

\bigskip
{\sc Acknowledgements.}  The author would like to thank Benjamin Br\"uck from Bielefeld, Germany,
for helful comments regarding the Coxeter group.  He is particularly grateful to Thomas Noll from Barcelona
since his thoughtful advice, in particular regarding the action of
interesting elements and subgroups of the Coxeter group on the Tonnetz,
has led to substantial improvements of the paper (which about doubled in size).



\begin{thebibliography}{999}

\bibitem{bb} Bj\"orner, A., Brenti, F.:
  Combinatorics of Coxeter Groups,
  Springer GTM {\bf 231} (2005).
\bibitem{cohn96} Cohn, R.:
  Maximally smooth cycles, hexatonic systems, and the analysis of late-romantic triadic progressions,
  Music Analysis, {\bf 15.1}  (1996), 9--40.
\bibitem{cohn97} Cohn, R.:
  Neo-Riemannian transformations, parsimonious trichords, and their Tonnetz representations,
  Journal of Music Theory {\bf 41:1} (1997), 1--66.
\bibitem{cohn} Cohn, R.:
  Introduction to neo-Riemannian  theory:
  a survey and a historical perspective,
  Journal of Music Theory {\bf 42:2} (1998), 167--198.
\bibitem{cohn-audacious} Cohn, R.:
  Audacious Euphony: Chromaticism and the Triad's Second Nature,
  Oxford Studies in Music Theory, Oxford University Press 2012.
\bibitem{cfs} Crans, A.S., Fiore, T.M., Satyendra, R.:
  Musical actions of dihedral groups,
  American Mathematical Monthly {\bf 116:6} (2009), 479--495.
\bibitem{klum} Klumpenhouwer, H.:
  Some remarks on the use of Riemann transformations,
  Music Theory Online {\bf 0(9)} (1994), 1--34.
\bibitem{rie} Riemann, H.:
  Skizze einer Neuen Methode der Harmonielehre,
  Breitkopf und Haertel. Leipzig 1880.
\end{thebibliography}
\end{document}